\documentclass{amsart}
\usepackage[latin1]{inputenc}
\usepackage{amsmath}
\usepackage{amsthm}
\usepackage{amsfonts}
\usepackage{amssymb}
\usepackage{enumerate}
\usepackage{mathrsfs}
\usepackage{graphicx}
\usepackage{hyperref}

\DeclareMathOperator{\orb}{orb}

\DeclareMathOperator{\Sing}{Sing}

\DeclareMathOperator{\Mat}{Mat}

\DeclareMathOperator{\GL}{GL}

\def\qa{{\boldsymbol{{G}}}}
\def\P{[0]}

\def\ZZ{\mathbb{Z}}
\def\BB{\mathbb{B}}
\def\NN{\mathbb{N}}
\def\CC{\mathbb{C}}
\def\Q{\mathbf{Q}}
\def\QQ{\mathbb{Q}}
\def\bQ{\mathbf{Q}}

\def\PP{\mathbb{P}}
\def\w{\omega}
\def\bw{\bar \omega}
\def\l{\ell}

\def\cC{\mathcal{C}}

\def\l{\ell}

\newcommand\bd{\mathbf{d}}

\newcommand\ba{\mathbf{a}}
\newcommand\bxi{\boldsymbol{\xi}}

\newtheorem{thm}{Theorem}[section]  
\newtheorem{prop}{Proposition}[section]
\newtheorem{cor}{Corollary}[section]
\newtheorem{lemma}{Lemma}[section]
\theoremstyle{remark}
\newtheorem{rem}{Remark}[section]
\newtheorem{notation}{Notation}[section]
\theoremstyle{definition}
\newtheorem{dfn}{Definition}[section]
\newtheorem{exam}{Example}[section]

\makeatletter
\let\c@lemma\c@thm
\let\c@prop\c@thm
\let\c@conj\c@thm
\let\c@cor\c@thm
\let\c@rem\c@thm
\let\c@dfn\c@thm
\let\c@notation\c@thm
\let\c@exam\c@thm
\makeatother

\def\makeautorefname#1#2{\expandafter\def\csname#1autorefname\endcsname{#2}}

\makeautorefname{thm}{Theorem}%
\makeautorefname{lemma}{Lemma}%
\makeautorefname{rem}{Remark}%
\makeautorefname{notation}{Notation}%
\makeautorefname{dfn}{Definition}%
\makeautorefname{exam}{Example}%
\makeautorefname{prop}{Proposition}%
\makeautorefname{cor}{Corollary}%

\title
[Local invariants on quotient singularities and a genus formula ...]
{Local invariants on quotient singularities and a genus formula for 
weighted plane curves}

\author[J.I. Cogolludo]{Jos{\'e} Ignacio Cogolludo-Agust{\'i}n}
\address{Departamento de Matem\'aticas, IUMA\\ 
Universidad de Zaragoza\\ 
C.~Pedro Cerbuna 12\\ 
50009 Zaragoza, Spain} 
\email{jicogo@unizar.es} 

\author[J. Martín]{Jorge Martín-Morales}
\address{Centro Universitario de la Defensa-IUMA\\ 
Academia General
Militar\\ 
Ctra. de Huesca s/n.\\ 
50090, Zaragoza, Spain} 
\email{jorge@unizar.es,jortigas@unizar.es}

\author[J. Ortigas]{Jorge Ortigas-Galindo}

\begin{document}

\thanks{All authors are partially supported by
the Spanish Ministry of Education MTM2010-21740-C02-02 and 
\emph{E15 Grupo Consolidado Geometr\'{\i}a} from the Gobierno de Aragón. 
The second author is also supported by FQM-333 from  
Junta de Andaluc{\'\i}a, and PRI-AIBDE-2011-0986 
\emph{Acción Integrada hispano-alemana}.
}


 \maketitle
\begin{abstract}
In this paper we extend the concept of Milnor fiber and Milnor number of a curve singularity 
allowing the ambient space to be a quotient surface singularity. A generalization of the local 
$\delta$-invariant is defined and described in terms of a $\bQ$-resolution of the curve singularity. 
In particular, when applied to the classical case (the ambient space is a smooth surface) one 
obtains a formula for the classical $\delta$-invariant in terms of a $\bQ$-resolution, 
which simplifies considerably effective computations. All these tools will finally allow for an 
explicit description of the genus formula of a curve defined on a weighted projective plane in terms 
of its degree and the local type of its singularities.
\end{abstract}

\section{Introduction}

In the present paper we focus on the study of curve germs on non-smooth surfaces, in particular on 
quotient surface singularities. Besides the interest on its own, this study will be useful to 
give a formula for the genus of curves in a weighted projective plane.

Hence, the goal of this paper is two-fold: first, define and investigate some of the properties 
of local invariants of curve germs on quotient singular surfaces such as the Milnor number, the 
$\delta^\w$-invariant, or the Noether formula, in terms of embedded $\Q$-resolutions. Second, apply
these invariants and formulas to obtain a global invariant such as the genus of a weighted plane 
curve in terms of their degree and the local $\delta^\w$-invariant of their singularities.

Let $X$ be a quotient singularity and $\cC=\{f=0\}$ a $\Q$-divisor on $X$. By means of the cyclic action,
one can canonically obtain a function $F$ on $X$ and thus define the Milnor fiber and Milnor number 
$\mu^\w$ of $F$ in a standard way (see~\cite{Le-Someremarks,STV-Milnornumbers}). Also, a definition of the 
local invariant $\delta^\w_0$ can be given in terms of the Milnor fiber of $(\cC,0)$ by means of the 
formula $\mu^\w=2\delta^\w -r^\w +1$, where $\mu^\w$ is the Milnor number of $(\cC,0)$ and $r^\w$ 
is the number of local branches of $\cC$ at~$0$. In the classical case ($X=\CC^2$) the invariant 
$\delta$ can be obtained from a resolution of the local singularity $(\cC,0)$ in $(\CC^2,0)$ as
\begin{equation}
\label{eq-delta-form}
\delta=\sum_{Q\prec 0} \frac{\nu_Q(\nu_Q -1)}{2},
\end{equation}
where $Q$ runs over all the infinitely near points of $0$ of the $\sigma$-process in the resolution and
$\nu_Q$ denotes the multiplicity of the strict transform of $\cC$ at~$Q$. 

A similar result is shown here by allowing $X$ to be a quotient surface singularity and the resolution
to be an embedded $\Q$-resolution (see \autoref{thm-delta}).

Consider an irreducible curve $\cC\subset \PP^2$ of degree $d$, it is a classical result 
(cf.~\cite{northcott-genus,Milnor-singular,Serre-groupes,Brieskorn-Knorrer-curves,casas-singularities}) 
that the genus of its normalization considered as a compact oriented Riemann surface (also denoted by $g(\cC)$) 
depends only on $d$ and the local type of its singularities as follows:
$$g(\cC)=g_d-\sum_{P\in \Sing(\cC)} \delta_P,$$
where $g_d=\frac {(d-1)(d-2)}{2}$. 

A natural question is to determine the genus of a weighted-projective curve from its degree and the local type 
of its singularities. This question is solved in \autoref{thm-main} (see also 
\cite{dolgachev-weighted,Orlik-Wagreich-isolated} for the quasi smooth case).

For the sake of completeness, a brief section (see \S\ref{sec-prelim}) to introduce the necessary definitions 
and notation is included.

\section{Preliminaries}
\label{sec-prelim}

Let us recall some definitions and properties on $V$-manifolds, weighted projective spaces, embedded $\bQ$-resolutions, 
and weighted blow-ups, see~\cite{dolgachev-weighted,AMO11a,AMO11b} for a more detailed exposition.

\subsection{$V$-manifolds and quotient singularities}
\mbox{}

\begin{dfn}
A $V$-manifold of dimension $n$ is a complex analytic space which admits an open
covering $\{U_i\}$ such that $U_i$ is analytically isomorphic to $B_i/G_i$ where
$B_i \subset \CC^n$ is an open ball and $G_i$ is a finite subgroup of $\GL(n,\CC)$.
\end{dfn}

\begin{thm}\label{th_Prill}{\rm (\cite{Prill67}).}~Let $G_1$, $G_2$ be small
subgroups of $\GL(n,\CC)$. Then $\CC^n/G_1$ is isomorphic to $\CC^n/G_2$ if
and only if $G_1$ and $G_2$ are conjugate subgroups. $\hfill \Box$
\end{thm}

\begin{notation}\label{notation_action}
For $\bd: = {}^t(d_1 \ldots d_r)$ we denote by
$\qa_{\bd} := \qa_{d_1} \times \cdots \times \qa_{d_r}$ a finite
abelian group written as a product of finite cyclic groups, that is, $\qa_{d_i}$
is the cyclic group of $d_i$-th roots of unity in $\CC$. Consider a matrix of weight
vectors 
\begin{align*}
A & := (a_{ij})_{i,j} = [\ba_1 \, | \, \cdots \, | \, \ba_n ] \in \Mat (r \times n, \ZZ), \\ 
\ba_j & := {}^t (a_{1 j}\dots a_{r j}) \in \Mat(r\times 1,\ZZ),
\end{align*}
and the action
\begin{equation*}
\begin{array}{cr}
( \qa_{d_1} \times \cdots \times \qa_{d_r} ) \times \CC^n  \longrightarrow  \CC^n,&\bxi_\bd = 
(\xi_{d_1}, \ldots, \xi_{d_r}), \\[0.15cm]
\big( \bxi_{\bd} , \mathbf{x} \big) \mapsto (\xi_{d_1}^{a_{11}} \cdot\ldots\cdot
\xi_{d_r}^{a_{r1}}\, x_1,\, \ldots\, , \xi_{d_1}^{a_{1n}}\cdot\ldots\cdot
\xi_{d_r}^{a_{rn}}\, x_n ), & \mathbf{x} = (x_1,\ldots,x_n).
\end{array}
\end{equation*}

Note that the $i$-th row of the matrix $A$ can be considered modulo $d_i$. The
set of all orbits $\CC^n / G$ is called ({\em cyclic}) {\em quotient space of
type $(\bd;A)$} and it is denoted by
$$
  X(\bd; A) := X \left( \begin{array}{c|ccc} d_1 & a_{11} & \cdots & a_{1n}\\
\vdots & \vdots & \ddots & \vdots \\ d_r & a_{r1} & \cdots & a_{rn} \end{array}
\right).
$$
The orbit of an element $\mathbf{x}\in \CC^n$ under this action is denoted
by $[\mathbf{x}]_{(\bd; A)}$ and the subindex is omitted if no ambiguity
seems likely to arise. Using multi-index notation
the action takes the simple form
\begin{eqnarray*}
\qa_\bd \times \CC^n & \longrightarrow & \CC^n, \\
(\bxi_\bd, \mathbf{x}) & \mapsto & \bxi_\bd\cdot\mathbf{x}:=(\bxi_\bd^{\ba_1}\, x_1, \ldots, \bxi_{\bd}^{\ba_n}\, x_n).
\end{eqnarray*}
\end{notation}

The quotient of $\CC^n$ by a finite abelian group is always isomorphic to a quotient space of type $(\bd;A)$, 
see~\cite{AMO11a} for a proof of this classical result. Different types $(\bd;A)$ can give rise to isomorphic quotient spaces.

\begin{exam}\label{Ex_quo_dim1} When $n=1$ all spaces $X(\bd;A)$ are
isomorphic to~$\CC$. It is clear that we can assume that $\gcd(d_i, a_{i})=1$.
If $r=1$, the map $[x] \mapsto x^{d_1}$ gives an isomorphism between $X(d_1; a_{1})$ and
$\CC$. 

Consider the case $r=2$. Note that
$\CC/(\qa_{d_1} \times \qa_{d_2})$ equals $(\CC/\qa_{d_1})/ \qa_{d_2}$.
Using the previous isomorphism, it is isomorphic to $X(d_2, d_1 a_2)$,
which is again isomorphic to $\CC$. By induction, we obtain the result for any~$r$.
\end{exam}

If an action is not free on $(\CC^{*})^n$ we can factor the group by the kernel of the action and
the isomorphism type does not change. This motivates the following definition.

\begin{def}\label{def_normalized_XdA}
The type $(\bd;A)$ is said to be {\em normalized} if the action is free on $(\CC^{*})^n$
and~$\qa_\bd$ is small as subgroup of $\GL(n,\CC)$.
By abuse of language we often say the space $X(\bd;A)$ is written in a normalized
form when we mean the type $(\bd;A)$ is normalized.
\end{def}

\begin{prop}
The space $X(\bd;A)$ is written in a normalized form if and only if the stabilizer
subgroup of $P$ is trivial for all~$P \in \CC^n$ with exactly $n-1$ coordinates
different from zero.

In the cyclic case the stabilizer of a point as above (with exactly $\,n-1$
coordinates different from zero) has order $\gcd(d, a_1, \ldots, \widehat{a}_i,
\ldots, a_n)$.
\end{prop}

It is possible to convert general types~$(\bd;A)$ into their normalized form.
\autoref{th_Prill} allows one to decide whether two quotient spaces are
isomorphic. In particular, one can use this result to compute the singular points
of the space $X(\bd;A)$.
In \autoref{Ex_quo_dim1} we have explained this normalization process in dimension one.
The two dimensional case is treated in the following example. 

\begin{exam}\label{X2}
All quotient spaces for $n=2$ are cyclic. The space $X(d;a,b)$ is written in a
normalized form if and only if $(d,a) = (d,b) = 1$, where $(a_1,\dots,a_n)$ denotes $\gcd(a_1,\dots,a_n)$. 
If this is not the case, one uses the isomorphism (assuming
$(d,a,b)=1$) $X(d;a,b) \rightarrow X \big( \frac{d}{(d,a)(d,b)}; \frac{a}{(d,a)},
\frac{b}{(d,b)} \big)$, $[ (x,y) ] \mapsto [ (x^{(d,b)},y^{(d,a)}) ]$
to convert it into a normalized one.
\end{exam}

\subsection{Weighted Blow-ups and Embedded $\mathbf{Q}$-Resolutions}\label{resolutions}
\mbox{}

\noindent
Classically an embedded resolution of $\{f=0\} \subset \CC^n$ is a proper map $\pi: X \to (\CC^n,0)$ 
from a smooth variety $X$ satisfying, among other conditions, that $\pi^{-1}(\{f=0\})$ is a normal 
crossing divisor. To weaken the condition on the preimage of the singularity we allow the new 
ambient space $X$ to contain abelian quotient singularities and the divisor $\pi^{-1}(\{f=0\})$ to 
have ``normal crossings'' over this kind of varieties. This notion of normal crossing divisor on 
$V$-manifolds was first introduced by Steenbrink in~\cite{Steenbrink77}.

We recall that in the class of $V$-manifolds, the abelian groups of Cartier and Weil divisors are not 
isomorphic. However the isomorphism can be achieved after tensoring by $\QQ$. Such divisors will be
referred to as \emph{$\Q$-divisors}. 

\begin{dfn}
Let $X$ be a $V$-manifold with abelian quotient singularities. A hypersurface
$D$ on $X$ is said to be with {\em $\Q$-normal crossings} if it is
locally isomorphic to the quotient of a union of coordinate hyperplanes under a group
action of type $(\bd;A)$. That is, given $x \in X$, there is an isomorphism of
germs $(X,x) \simeq (X(\bd;A), [0])$ such that $(D,x) \subset (X,x)$ is identified
under this morphism with a germ of the form
$$
\big( \{ [\mathbf{x}] \in X(\bd;A) \mid x_1^{m_1} \cdot\ldots\cdot x_k^{m_k} = 0 \},
[(0,\ldots,0)] \big).
$$
\end{dfn}

\begin{def}\label{Qresolution}
Let $M = \CC^{n+1} / G$ be an abelian quotient space. Consider $H \subset M$ an analytic 
subvariety of codimension one. An embedded $\Q$-resolution of $(H,0) \subset (M,0)$ 
is a proper analytic map~$\pi: X \to (M,0)$ such that:
\begin{enumerate}
\item $X$ is a $V$-manifold with abelian quotient singularities.
\item $\pi$ is an isomorphism over $X\setminus \pi^{-1}(\Sing(H))$.
\item $\pi^{-1}(H)$ is a hypersurface with $\Q$-normal crossings on $X$.
\end{enumerate}
\end{def}

\begin{rem}
\label{rem-nc-qr}
In some cases, one needs to consider a stronger condition on $\Q$-resolutions, namely, the 
strict transform of $H$ does not contain any singular points of $X$.
This can always be achieved by blowing up eventually once more the strict preimage of the hypersurface.
Such an embedded resolution will be referred to as a \emph{strong $\Q$-resolution}.
\end{rem}

Usually one uses weighted or toric blow-ups with smooth center as a tool for finding embedded 
$\Q$-resolutions. Here we only discuss briefly the surface case. Let $X$ be an analytic 
surface with abelian quotient singularities. Let us define the weighted blow-up 
$\pi: \widehat{X} \to X$ at a point $P\in X$ with respect to $\w = (p,q)$. 
We distinguish two different situations.

\vspace{0.20cm}

{\bf (i) The point $P$ is smooth.} Assume $X = \CC^2$ and $\pi = \pi_{\w}: \widehat{\CC}^2_{\w} \to \CC^2$ 
the weighted blow-up at the origin with respect to $\w =(p,q)$,
\[
\widehat{\CC}^2_{\w}:=\{((x,y),[u:v]_{\w})\in\CC^2\times\PP^1_{\w}\mid (x,y)\in\overline{[u:v]}_{\w}\}.
\]

Here the condition about the closure means that $\exists t \in \CC$ , $x=t^pu$, $y =t^qv$.
The new ambient space is covered as $\widehat{\CC}^2_{\w} = U_1 \cup U_2 = X(p;-1,q) \cup X(q;\,p,-1)$ 
and the charts are given by
\begin{center}
$\begin{array}{ccc|ccc}
X(p;-1,q) & \longrightarrow & U_1, & X(q;\,p,-1) & \longrightarrow & U_2, \\[0.10cm]
\,[(x,y)] & \mapsto & ((x^p,x^q y),[1:y]_{\w}); & [(x,y)] & \mapsto & ((x y^p, y^q),[x:1]_{\w}).
\end{array}$
\end{center}

The exceptional divisor $E=\pi_{\w}^{-1}(0)$ is isomorphic to $\PP^1_{\w}$ which is in turn 
isomorphic to~$\PP^1$ under the map $[x:y]_{\w} \mapsto [x^q:y^p]$. The singular points of 
$\widehat{\CC}^2_{\w}$ are cyclic quotient singularities located at the exceptional divisor. 
They actually coincide with the origins of the two charts.

\vspace{0.20cm}

{\bf (ii) The point $P$ is of type $(d;a,b)$.} Assume that $X=X(d;a,b)$. The group $\qa_d$ 
acts also on $\widehat{\CC}^2_{\w}$ and passes to the quotient yielding a map 
$\pi = \pi_{(d;a,b),\w}: \widehat{X(d;a,b)}_{\w} \to X(d;a,b)$, where by definition 
$\widehat{X(d;a,b)}_{\w}:= \widehat{\CC}^2_{\w} / \qa_d$. The new space is covered as
\begin{equation}
\label{eq-charts} 
\widehat{X(d;a,b)}_{\w} = \widehat{U}_1 \cup \widehat{U}_2 = 
X \left( \displaystyle\frac{pd}{e}; 1, \frac{-q+a' pb}{e} \right) \cup X 
\left( \displaystyle\frac{qd}{e}; \frac{-p+b' qa}{e}, 1\right)
\end{equation}
with $a'a=b'b\equiv 1 \mod (d)$ and $e = \gcd(d,pb-qa)$. The charts are given by

\begin{center}
$\begin{array}{c|c}
X \left( \displaystyle\frac{pd}{e}; 1, \frac{-q+a' pb}{e} \right)  \ \longrightarrow \ 
\widehat{U}_1, &X \left( \displaystyle\frac{qd}{e}; \frac{-p+b' qa}{e}, 1
\right) \ \longrightarrow \ \widehat{U}_2, \\[0.2cm] \,\big[ (x^e,y) \big] \mapsto 
\big[ ((x^p,x^q y),[1:y]_{\w}) \big]_{(d;a,b)} & \big[ (x,y^e) \big] \mapsto 
\big[ ((x y^p, y^q),[x:1]_{\w}) \big]_{(d;a,b)}.
\end{array}$
\end{center}

The exceptional divisor $E = \pi_{(d;a,b),\w}^{-1}(0)$ is identified with 
$\PP^1_{\w}(d;a,b) := \PP^1_{\w}/\qa_{d}$. Again the singular points are cyclic 
and correspond to the origins of the two charts.

\vspace{0.15cm}

\begin{prop}\label{formula_self-intersection}
Let $X$ be a surface with abelian quotient singularities. Let $\pi: \widehat{X} \to X$ 
be the weighted blow-up at a point $P$ of type $(d;a,b)$ with respect to $\w=(p,q)$. 
Assume $(d,a)=(d,b)=(p,q)=1$ and write $e = \gcd(d,pb-qa)$.

Consider $C$ and $D$ two $\Q$-divisors on $X$, denote by $E$ the exceptional divisor of $\pi$, 
and by $\widehat{C}$ (resp.~$\widehat{D}$) the strict transform of $C$ (resp.~$D$). Let $\nu_{C,P}$ 
(resp.~$\nu_{D,P}$) be the $(p,q)$-multiplicity of $C$ (resp $D$) at $P$, (defined such that 
$\nu_{x,P}=p$ and $\nu_{y,P}=q$). Then the following equalities hold:
\begin{enumerate}[$(1)$]
 \item\label{formula_self-intersection1} 
 $\displaystyle \pi^{*}(C) = \widehat{C} + \frac{\nu_{C,P}}{e} E$
 \item\label{formula_self-intersection2} 
 $\displaystyle E \cdot \widehat{C} = \frac{e \nu_{C,P}}{p q d}$
 \item\label{formula_self-intersection3} 
 $\displaystyle (C \cdot D)_P=\widehat{C} \cdot \widehat{D}+ \frac{\nu_{C,P}\nu_{D,P}}{p q d}$.
\end{enumerate}
\end{prop}

\subsection{\textbf{Local intersection number on $X(d;a,b)$}}
\label{computation_local_number_V-surface}
\mbox{}

\noindent
Denote by $X$ the cyclic quotient space $X(d;a,b)$ and
consider two divisors $D_1 = \{ f_1 = 0 \}$ and $D_2 = \{ f_2 = 0 \}$ given by
$f_1,f_2\in\CC\{x,y\}$ reduced.
Assume that, $(d; a, b)$ is normalized, $D_1$ is
irreducible, $f_1$ induces a function on $X$, and 
finally that $D_1 \nsubseteq D_2$.

Then as Cartier divisors $D_1 = \{(X,f_1) \}$, $D_2 = \frac{1}{d} \{(X,f_2^d)\}$.
The local number $(D_1 \cdot D_2)_{[P]}$ at the singular point $P$ is defined as
$$
(D_1 \cdot D_2)_{[P]} =
\displaystyle \frac{1}{d} \dim_{\CC} \left( \frac{\CC\{x,y\}^{\qa_d}}{\langle f_1,
f_2^d \rangle} \right),
$$
where $\CC\{x,y\}^{\qa_d}$ is the local ring of functions at~$P$.

Analogously, if $f_1$ does not define a function on $X$, for computing the
intersection number at $P$ one substitutes $f_1$ by $f_1^d$ and divides
the result by $d$.

\begin{exam}
\label{exam-nu-nc}
Let $x_1$ and $x_2$ be the local coordinates of the axes in $M:=X(d;a,b)$. 
Consider $X_i:=\{(M,x_i)\}$ the $\Q$-divisors associated with them.
Note that
$$
(X_1\cdot X_2)_0=\frac{1}{d}.
$$
\end{exam}

\subsection{Weighted projective plane}\label{sec-wpp}
\mbox{}

\noindent
The main reference that has been used in this subsection is \cite{dolgachev-weighted}.
Here we concentrate our attention on describing the analytic structure and singularities.

Let $\w:=(\w_0,\w_1,\w_2)$ be a weight vector, that is, a finite set of coprime positive
integers. 
There is a natural action of the multiplicative group $\CC^{*}$ on
$\CC^{3}\setminus\{0\}$ given by
$$
  (x_0,x_1,x_2) \longmapsto (t^{w_0} x_0,t^{w_1} x_1,t^{\w_2} x_2).
$$

The set of orbits $\frac{\CC^{3}\setminus\{0\}}{\CC^{*}}$ 
under this action is denoted by $\PP^2_\w$ and it is called the {\em weighted projective plane} of type $\w$. 
The class of a nonzero element $(X_0,X_1,X_2)\in \CC^{3}$ 
is denoted by $[X_0:X_1:X_2]_\w$ and the weight vector is omitted if no ambiguity seems likely to arise.
When $(\w_0,\w_1,\w_2)=(1,1,1)$ one obtains the usual projective space
and the weight vector is always omitted. For $\mathbf{x}\in\CC^{3}\setminus\{0\}$,
the closure of $[\mathbf{x}]_\w$ in $\CC^{3}$ is obtained by adding the origin and it
is an algebraic curve.

Consider the decomposition
$
\PP^2_\w = U_0 \cup U_1 \cup U_2,
$
where $U_i$ is the open set consisting of all elements $[X_0:X_1:X_2]_\w$
with $X_i\neq 0$. The map
$$
  \widetilde{\psi}_0: \CC^2 \longrightarrow U_0,\quad
\widetilde{\psi}_0(x_1,x_2):= [1:x_1:x_2]_\w
$$
defines an isomorphism $\psi_0$ if we replace $\CC^2$ by
$X(\w_0;\, \w_1,\w_2)$.
Analogously, $X(\w_1;\,\w_0,\w_2) \cong U_1$ and $X(\w_2;\,\w_0,\w_1) \cong U_2$
under the obvious analytic map.

\begin{rem}(Another way to describe $\PP^2_{\w}$).~Let $\PP^2$ be the classical projective space and 
$\qa_\w = \qa_{\w_0}\times \qa_{\w_1}\times \qa_{\w_2}$ the product of cyclic groups. Consider the group action
$$
\begin{array}{rcl}
\qa_\w \times \PP^2 \quad & \longrightarrow & \quad \PP^2, \\[0.15cm]
\big( (\xi_{\w_0},\xi_{\w_2},\xi_{\w_2}), [X_0:X_1:X_2] \big) & \mapsto &
[\xi_{\w_0}X_0:\xi_{\w_1}X_1:\xi_{\w_2}X_2].
\end{array}
$$

Then the set of all orbits $\PP^2/\qa_\w$ is isomorphic to the weighted projective plane of type $\w$ and 
the isomorphism is induced by the branched covering
\begin{equation}\label{covering}
   \PP^2\ni [X_0:X_1:X_2] \overset{\phi}{\longmapsto}  [X_0^{\w_0}:X_1^{\w_1}:X_2^{\w_2}]_\w \in\PP^2_\w.
\end{equation}

Note that this branched covering is unramified over
$$
\PP^2_{\w} \setminus \{ [X_0,X_1,X_2]_{\w} \mid X_0\cdot X_1\cdot X_2 = 0 \}
$$
and has $\bar{\w}=\w_0\cdot\w_1\cdot\w_2$ sheets. Moreover, the covering respects the coordinate axes.
\end{rem}

\begin{exam}
Let $\phi:\PP^2 \to \PP^2_{\w}$ be the branched covering defined above with weights $\w = (1,2,3)$. 
For instance, the preimage of $[1:1:1]_{\w}$ consists of 6 points, namely the set 
$\{ [1:\xi_2:\xi_3] \in \PP^2 \mid \xi_2 \in \qa_2,\ \xi_3 \in \qa_3 \}$.
\end{exam}

The following result is well known.

\begin{prop}\label{propPw}
Let $d_0 := \gcd (\w_1,\w_2)$, $d_1 := \gcd (\w_0,\w_2)$ , $d_2 := \gcd (\w_0,\w_1)$,
 $e_0:= d_1\cdot d_2$,  $e_1:= d_0\cdot d_1$, $e_2:= d_0\cdot d_1$ 
and $p_i:=\frac{\w_i}{e_i}$. The following map is an isomorphism:
$$
\begin{array}{rcl}
\PP^2 \big(\w_0,\w_1,\w_2\big) & \longrightarrow & \PP^2(p_0,p_1,p_2), \\[0.15cm]
\,[X_0:X_1:X_2] & \mapsto &
\big[\,X_0^{d_0}:X_1^{d_1}:X_2^{d_2}\,\big].
\end{array}
$$
\end{prop}

\begin{rem}
 Note that, due to the preceding proposition, one can always assume the weight
vector satisfies that $(\w_0,\w_1,\w_2)$ are pairwise relatively prime numbers. 
Note that following a similar argument, $\PP^1_{(\w_0,\w_1)} \cong \PP^1$.
\end{rem}

\begin{prop}[Bézout's Theorem on $\PP^2_{\w}$]\label{thm-bezout-P2w}
The intersection number of two $\Q$-divisors, $D_1$ and $D_2$ on $\PP^2_{\w}$ without common components is
\vspace{0.20cm}
$$
\displaystyle 
D_1 \cdot D_2 =\sum_{P\in D_1\cap D_2} (D_1\cdot D_2)_{[P]}= 
  \frac{1}{\bar{\w}} \deg_{\w}(D_1) \deg_{\w}(D_2) \in \mathbb{Q},
$$
where $\bar{\w}=\w_0\w_1\w_2$.
\end{prop}

\section{Milnor fibers on quotient singularities}
\label{sec-milnor}

Our purpose in this section is to provide a definition for the Milnor fiber of the germ of an 
isolated hypersurface $(f,\P)$ singularity defined on an abelian $V$-surface $X(d;a,b)$. 

In the classical case, if $(\cC=\{f=0\},0)\subset (\CC^2,0)$ defines a local singularity, the Milnor 
fiber is defined as $F_t=\{f=t\}$ and it satisfies $\chi (F_t)=r-2\delta$, where $r$ is the number
of local branches of $(\cC,0)$. Note that this cannot be extended directly to the case 
$\cC\subset X(d;a,b)$ because, in general, the germ $(f,\P)$ does not define a function on $X(d;a,b)$.
However, $F:=\prod_{g\in \qa} f^g=uf^d$ ($u$ a unit) is a well-defined function on $X(d;a,b)$ and hence the set $\{F=t\}$ 
is also well defined and invariant under the action of the cyclic group~$\qa_d$. One can offer the following 
alternative definition for the Milnor fiber of~$\cC$.

\begin{dfn}
Let $\cC=\{f=0\} \subset X(d;a,b)$ be a curve germ. 
The \emph{Milnor fiber} $F_t^\w$ of $(\cC,\P)$ is defined as follows,
$$F_t^\w:=\{ F=t\}/{\qa_d}.$$

The \emph{Milnor number} $\mu^\w$ of $(\cC,\P)$ is defined as follows,
$$\mu^\w:=1-\chi^{\orb}(F_t^\w).$$
\end{dfn}

The symbol $\chi^{\orb}(M)$ denotes the orbifold Euler characteristic of $M\subset X=\CC^2/\qa_d$ as a 
subvariety of a quotient space which carries an orbifold structure. Note that one can also consider 
$\cC$ as a germ in $(\CC^2,0)$ in which case, $\chi^{\orb}(F_t^\w)=\frac{1}{d}\chi(F_t)$, where 
$F_t=\{f=t\}\subset \CC^2$. Therefore, $1-\mu^\w=\frac{1}{d}-\frac{\mu}{d}$, which implies
\begin{equation}
\label{eq-milnor}
\mu^\w=\frac{d-1}{d}+\frac{\mu}{d}.
\end{equation}

Note the difference between this definition and 
\cite[Definition~2.10]{ABLM-milnor-number}.

Also note that $\mu^\w(x)=\mu^\w(y)=\frac{d-1}{d}$, which extends immediately to all local 
curves that can become an axis after an action-preserving change of coordinates. In other 
words, 

\begin{dfn}
\label{def-qsmooth}
The curve $\{f=0\}\subset X(d;a,b)$ is called a \emph{$\Q$-smooth} curve if there exists 
$g\in \CC\{x,y\}$ such that $\CC\{f,g\}^\qa=\CC\{x,y\}^\qa$.
\end{dfn}

\begin{cor}
$$\{f=0\} \text{ is a } \Q \text{-smooth curve} \ \Longleftrightarrow \ \mu^\w(f)=\frac{d-1}{d}.$$ 
\end{cor}

\section{Local invariants on quotient singularities}
\label{sec-delta}

\subsection{Noether's Formula}
\mbox{}

\noindent
In this section we present a version of Noether's formula for curves on quotient singularities
and $\Q$-resolutions. During this section we will use the notation introduced in 
\S\ref{resolutions}. As an immediate consequence of 
\autoref{formula_self-intersection}\eqref{formula_self-intersection3} one has the following formula,

\begin{equation}\label{eqlemanoether}
(C \cdot D)_{\P}= \frac{\nu_{C}\nu_{D}}{p q d}+\sum_{Q\prec {\P}}(\widehat{C} \cdot \widehat{D})_Q,
\end{equation}
where $Q$ runs over all infinitely near points to ${\P}\in X(d;a,b)$ after a weighted $(p,q)$-blow-up.

By induction, using formula~\eqref{eqlemanoether}, one can prove a Noether's formula 
for $\Q$-divisors on quotient surface singularities.

\begin{thm}[Noether's Formula]\label{thm-noether}
Consider $C$ and $D$ two germs of $\Q$-divisors without common components at ${\P}$ a quotient surface 
singularity. Then the following formula holds:

$$(C \cdot D)_{\P}= \sum_{Q\prec {\P}}\frac{\nu_{C,{Q}}\nu_{D,{Q}}}{p q d},$$
where $Q$ runs over all the infinitely near points of $\P$ and $Q$ appears after a blow-up of type
$(p,q)$ of the origin in $X(d;a,b)$. 
\end{thm}

\begin{rem}
\label{rem-noether}
Note that $p,q,d,a,$ and $b$ in \autoref{thm-noether} depend on $Q$ and its predecesor.
\end{rem}

\subsection{Definition of the $\delta$ invariant on quotient singularities}
\mbox{}

\noindent
In this section the local invariant $\delta^\w$ for curve singularities on $X(d;a,b)$ is defined.

\begin{dfn}
Let $C$ be a reduced curve germ at $\P\in X(d;a,b)$, then we define 
$\delta^\w$ as the number verifying 
\begin{equation}
 \label{eq-def-delta}
\chi^{\orb}(F_t^\w)=r^\w - 2 \delta^\w,
 \end{equation}
where $r^\w$ is the number of local branches of $C$ at ${\P}$, $F_t^\w$ denotes its Milnor fiber,
and $\chi^{\orb}(F_t^\w)$ denotes the orbifold Euler characteristic of~$F_t^\w$.
\end{dfn}

Using the same argument as in~\eqref{eq-milnor}, one can check that
\begin{equation}
\label{eq-deltaw}
\delta^\w=\frac{1}{d}\delta + \frac{1}{2} \left( r^\w-\frac{r}{d}\right)
\end{equation}
where $\delta$ denotes the classical $\delta$-invariant of $C$ as a germ in~$(\CC^2,0)$ and
$r$ is the number of local branches of $C$ in~$(\CC^2,0)$.

\begin{rem}
At this point it is worth mentioning that $r^\w$ and $r$ do not necessarily coincide. For instance,
$(x^2-y^4)$ defines an irreducible curve germ in $X(2;1,1)$ (hence $r^\w=1$), but it is not irreducible 
in $\CC^2$ (where $r=2$). One can check that $\delta^\w=1$, $\delta=2$ (see \autoref{deltapqab}), which
verify \eqref{eq-deltaw}.
\end{rem}

The purpose of this section is to give a recurrent formula for $\delta^\w$ based on a $\Q$-resolution
of the singularity. For technical reasons it seems more natural to use strong $\Q$-resolutions 
(see \autoref{rem-nc-qr}) for the statement, but this is no restriction (see \autoref{rem-qresolution}).

\begin{thm}
\label{thm-delta}
Let $(C,{\P})$ be a curve germ on an abelian quotient surface singularity. Then
\begin{equation}
\label{eq-delta}
\delta^\w=\frac{1}{2}\sum_{Q\prec {\P}} \frac{\nu_Q}{dpq}\left(\nu_Q-p-q+e \right),
\end{equation}
where $Q$ runs over all the infinitely near points of a strong $\Q$-resolution of $(C,{\P})$, 
$Q$ appears after a $(p,q)$-blow-up of the origin of $X(d;a,b)$, and 
$e:=\gcd(d,aq-bp)$.
\end{thm}

A similar comment to \autoref{rem-noether} applies to \autoref{thm-delta}.

\begin{proof}
Since we want to proceed by induction, let us assume $\P\in X(d;a,b)$. 
After a $(p,q)$-blow-up of $\P$ there are three types of infinitely near points to $\P$, 
namely, $P_1$ (resp.~$P_2$) a point on the surface of local type 
$X(\frac{pd}{e};1,\frac{-q+a' pb}{e})$ (resp. $X(\frac{qd}{e};\frac{-p+b' qa}{e},1)$)
and $P_3,\dots,P_n$ smooth points on the surface (see \eqref{eq-charts}). 
Outside the neighborhoods $\BB_i$ of the points $P_i$, the preimage of the Milnor 
fiber is a covering of degree~$\frac{\nu}{e}$ over $E$.

Therefore
$$\chi^{\orb}(F_t^\w)=\chi^{\orb}(E\setminus \{P_1,\dots,P_n\}) \frac{\nu}{e} + 
\sum_i \chi^{\orb}(\BB_i\cap \{x^{\frac{d\nu}{e}}\tilde F_i=t\}),$$
since the intersection is glued over disks, whose Euler characteristic is zero.
Here $\tilde F_i$ denotes the strict preimage of $F$ at~$P_i$.

In order to compute $\chi^{\orb}(\BB_i\cap \{x^{\frac{d\nu}{e}}\tilde F_i=t\})$
we have to distinguish three cases: $P_1,P_2$, and $P_i$ ($i=3,\dots,n$).

Assume first that $P=P_i$, $i=3,\dots,n$, then one can push $\tilde F_i$ into 
$\tilde F'_i$ in a direction transversal to $E$ as in \autoref{fig-push}, then 
$\chi^{\orb}(\BB_i\cap \{x^{\frac{d\nu}{e}}\tilde F_i=t\})=
\chi((\tilde F'_i)_t) - (E,\tilde F_i)_{P_i}$.

In case $P=P_1\in X(\frac{pd}{e};1,\frac{-q+a' pb}{e})$, one has 
$\chi^{\orb}(\BB_1\cap \{x^{\frac{d\nu}{e}}\tilde F_1=t\})=
\frac{e}{pd}\chi(\BB_1\cap \{x^{\frac{d\nu}{e}\frac{1}{d}}\tilde f_1=t\})$,
where $\tilde f_1$ is the preimage of $\tilde F_1$ in $\CC^2$. Therefore,
after applying the pushing strategy, 
$\chi(\BB_1\cap \{x^{\frac{\nu}{e}}\tilde f_1=t\})=
\frac{\nu}{e}\left( 1 - (E,f_1)_{P_1}\right)$,
and hence
$$
\chi^{\orb}(\BB_1\cap \{x^{\frac{d\nu}{e}}\tilde F_1=t\})=
\frac{\nu}{e}\left( \frac{e}{pd} - (E,\tilde F_1)_{P_1}\right).
$$
One obtains an analogous formula for $P_2$. Adding up all the terms and applying
\autoref{formula_self-intersection}\eqref{formula_self-intersection2} one obtains:
$$
\chi^{\orb}(F_t^\w)=\frac{\nu}{dp}+\frac{\nu}{dq}-\frac{e\nu}{dpq}(1+\frac{\nu}{e})
+\sum_i \chi^{\orb}(\tilde F_i)=
-\frac{\nu}{dpq}(\nu-p-q+e)+\sum_i \chi^{\orb}(\tilde F_i).
$$
The formula follows by induction since, after a strong $\Q$-resolution, 
$\sum_i \chi^{\orb}(\tilde F_i)=r^\w$ and hence
\begin{equation*}
\chi^{\orb}(F_t^\w)=
r^\w - \sum \frac{\nu}{dpq}(\nu-p-q+e)=r^\w-2\delta^\w. 
\end{equation*}
\begin{figure}
\begin{center}
\includegraphics[scale=1]{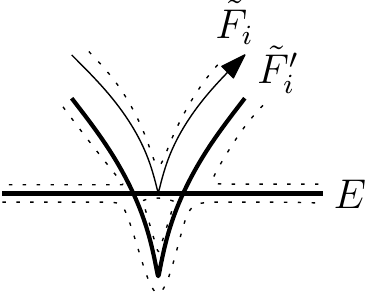}
\caption{Pushing $\tilde F_i$}
\label{fig-push}
\end{center}
\end{figure}
\end{proof}

\begin{exam}\label{deltapqab}
Assume $x^p-y^q=0$ defines a curve on a surface singularity of type $X(d;a,b)$. 
Note that a simple $(q,p)$-blow-up will be a (strong) $\Q$-resolution of the singular point. 
Therefore, using \autoref{thm-delta} one obtains
$$
\delta^\w=\frac{\nu}{2dqp}(\nu-q-p+e).
$$
Note that $\nu=pq$. Also, since $x^p-y^q=0$ defines a set of zeros in $X(d;a,b)$
by hypothesis, this implies $ap\equiv bq \mod d$ and hence $e:=\gcd(d,ap-bq)=d$. Thus
\begin{equation}
\label{eq-pq}
\delta^\w=\frac{(pq-p-q+d)}{2d}.
\end{equation}
Note that this provides a direct proof, for the classical case $(d=1)$, that 
$$\delta=\frac{(p-1)(q-1)}{2}$$
for a singularity of type $x^p-y^q$ in $(\CC^2,0)$.

Another direct consequence of~\eqref{eq-pq} is that 
\begin{equation}
\label{eq-delta-qsmooth}
\delta^\w(x)=\delta^\w(y)=\frac{d-1}{2d}
\end{equation}
and the same formula holds for any $\Q$-smooth curve (see \autoref{def-qsmooth}).
\end{exam}

\begin{rem}
\label{rem-qresolution}
By~\eqref{eq-delta-qsmooth}, a $\Q$-resolution is enough to obtain $\delta^\w$. Note that 
a $\Q$-resolution end when the branches are separated and each strict transform is a 
$\Q$-smooth curve. Therefore if \eqref{eq-delta} is applied for a $\Q$-resolution, one
needs to add $\frac{d_i-1}{2d_i}$ for each local branch $\gamma_i\subset X(d_i;a_i,b_i)$, 
$i=1,\dots,r^\w$. 
\end{rem}

\begin{cor}
Let $C$ and $D$ be two reduced $\Q$-divisors at ${\P}\in X(d;a,b)$ without common components. 
Then
$$\delta^\w(C\cdot D)=\delta^\w(C)+\delta^\w(D)+(C \cdot D)_{\P}.$$
\end{cor}

\begin{proof}
 One has,
\begin{equation}
\label{ecu}
\array{c}
\nu_{C\cdot D} (\nu_{C\cdot D}-p-q+e)=(\nu_{C}+\nu_{D})(\nu_{C}+\nu_{D}-p-q+e)=\\[0.2cm]
=\nu_{C}(\nu_{C}-p-q+e)+\nu_{D}(\nu_{D}-p-q+e)+2\nu_{C}\nu_{D}.
\endarray
\end{equation}

Dividing \eqref{ecu} by $2dpq$ and making the addition over all the infinitely near points to ${\P}$ one obtains,
\begin{equation*}
\delta^\w(C\cdot D)=\delta^\w(C)+\delta^\w(D)+\sum_{Q\prec {\P}}
\frac{\nu_{C,{Q}}\nu_{D,{Q}}}{p q d}=\delta^\w(C)+\delta^\w(D)+(C \cdot D)_{\P}. 
\end{equation*}
\end{proof}

\subsection{The $\delta^\w$ invariant as the dimension of a vector space}
\mbox{}

\noindent
In the classical case this invariant can be interpreted as the dimension of a vector space. 
Since $\delta^\w$ is in general a rational number, a similar result can only be expected in 
certain cases, namely, when associated with Cartier divisors. This section is devoted to proving 
this fact. 

Let us start with the following constructive result which allows one to see any singularity 
on the quotient $X(d;a,b)$ as the strict transform of some $\{g=0\} \subset \CC^2$ after 
performing a certain weighted blow-up.

\begin{lemma}\label{lema1}
Let $f \in \CC\{x\}[y]$ defining an analytic function germ on $X(p;-1,q)$, $\gcd(p,q)=1$, 
such that $x \nmid f$. Then there exist $g \in \CC\{x\}[y]$ with $x \nmid g$ and $\nu \in \NN$ multiple of $p$
such that $g(x^p,x^q y) = x^{\nu} f(x,y)$. Moreover, $f$ is reduced (resp.~irreducible) 
if and only if $g$~is.
\end{lemma}

\begin{proof} 
Let us write $f$ as a finite sum of monomials, $f=\sum_{i,j} a_{ij} x^i y^j$. Since every monomial of 
$f$ is $\qa_p$-invariant, $p$ divides $-i+qj$ for all $i,j$. Consider
$$
g_1(x,y) = \sum_{i,j} a_{ij} x^{\frac{i-qj}{p}} y^j \ \in \ \CC\{x^{\pm 1}\}[y]
$$
and take the minimal non-negative integer $\l$ such that $g(x,y) := x^\l g_1(x,y)$ is an element of $\CC\{x\}[y]$. 
Note that $\l$ exists because $f$ is a polynomial in $y$.

Then $g(x^p,x^q y) = x^{p \l} f(x,y)$ and the minimality of $\l$ ensures that $x \nmid g$. 
The final part of the statement is a consequence of the fact that the $(p,q)$-blowing-up is a 
birational morphism.
\end{proof}

Let $I_{p,q} := \overline{\langle x^q, y^p \rangle} \subset \CC \{ x, y \}$ be the 
integral closure of $\langle x^q, y^p \rangle$ in $\CC \{ x, y \}$ with $\gcd(p,q)=1$. 
Note that the integral closure of a monomial ideal is well understood. In fact, the exponent set of the integral 
closure of a monomial ideal equals all the integer lattice points in the convex hull of the exponent set of such 
ideal, (see for instance~\cite[Proposition~1.4.6]{huneke-swanson}). 
In our case, $I_{p,q}^n = \overline{\langle x^q, y^p \rangle^n}$ and the following property 
holds for any $h \in \CC\{x,y\}$,
\begin{equation}\label{property_Ipq}
h \in I_{p,q}^n \ \Longleftrightarrow \ \nu_{p,q}(h) \geq pqn.
\end{equation}

\begin{lemma}\label{lemma_dim_I+f}
Assume $f \in \CC\{x,y\}$ has $(p,q)$-order $\nu$ a multiple of $pq$. Then, for all 
$n \geq \frac{\nu}{pq}$, one has
$$
\dim_\CC \left( \frac{\CC\{x,y\}}{I_{p,q}^n+ \langle f \rangle} \right) = n \nu - \frac{\nu (\nu - p -q +1)}{2pq}.
$$
\end{lemma}

\begin{proof}
Since $pq \,|\, \nu$, by~\eqref{property_Ipq}, the multiplication by $f$ gives rise to the short 
exact sequence for all $n \geq \frac{\nu}{pq}$,
\begin{equation*}
0 \longrightarrow \frac{\CC\{x,y\}}{I_{p,q}^{n-\frac{\nu}{pq}}} \overset{\cdot f}{\longrightarrow} 
\frac{\CC\{x,y\}}{I_{p,q}^n} \longrightarrow \frac{\CC\{x,y\}}{I_{p,q}^n+ \langle f \rangle} \longrightarrow 0.
\end{equation*}

Note that, by virtue of Pick's Theorem and~\eqref{property_Ipq}, the dimension of the second vector 
space above is given by the formula
\begin{equation*}
F(n) := \dim_\CC \left( \frac{\CC\{x,y\}}{I_{p,q}^n} \right) = pq\, \frac{n(n+1)}{2}-\frac{(p-1)(q-1)}{2}\, n.
\end{equation*}

Finally, the required dimension is $F(n) - F(n-\frac{\nu}{pq})$ and the claims follow.
\end{proof}

Although the following results can be stated in a more general setting, considering a direct 
sum of vector spaces in the ring $R^1$ (see below), we proceed in this way for the sake of simplicity.

Let $f:(\CC^2,0)\to (\CC,0)$ be an analytic function germ and $C:= \{ f = 0 \}$. Consider the weighted 
blow-up at the origin with $\gcd(p,q)=1$. Assume the exceptional divisor $E$ and the strict transform 
$\widehat{C}$ intersect just at the origin of the first chart $X(p;-1,q)$ and the latter divisor is 
given by a well-defined function $\widehat{f}$ on the quotient space. Also denote by 
$R = \frac{\CC\{x,y\}}{\langle f \rangle}$ and $R^{1} = \frac{\CC\{x,y\}^{\qa_p}}{\langle \widehat{f} \rangle}$ 
their corresponding local rings.

\begin{rem}\label{weierstrass_remark}
The Weierstrass division theorem states that given $f,g \in \CC\{x,y\}$ with $f$ $y$-general of order $k$, 
there exist $q \in \CC\{x,y\}$ and $r \in \CC\{x\}[y]$ of degree in $y$ less than or equal to $k-1$, both 
uniquely determined by $f$ and $g$, such that $g = q f + r$. The uniqueness and the linearity of the action 
ensure that the division can be performed in $\CC\{x,y\}^{\qa_d}$, i.e.~if $f$ and $g$ are $\qa_d$-invariant, 
then so are $q$ and~$r$. In other words, the Weierstrass preparation theorem still holds in $\CC\{x,y\}^{\qa_d}$.
\end{rem}

Assume $f$ is a Weierstrass polynomial in $y$ of degree $b$ with $\nu:=\nu_{p,q}(f) = qb$. Then its strict 
transform $\widehat{f}$ is again a Weierstrass polynomial in $y$ of the same degree. Classical arguments using 
the Weierstrass division theorem, see for instance \cite[Theorem~1.8.8]{casas-singularities}, provide the following 
isomorphisms $R \simeq \frac{\CC\{x\}[y]}{\langle f \rangle}$ and $R^1 \simeq 
\frac{(\CC\{x\}[y])^{\qa_p}}{\langle \widehat{f} \rangle}$, which allow one to prove that the pull-back morphism
\begin{eqnarray*}
\varphi: \ R & \longrightarrow & R^1 \\
h(x,y) & \mapsto & h(x^p,x^qy)
\end{eqnarray*}
is in fact injective. Hereafter $R$ is identified with a subring of $R^1$ and thus one just simply writes $R \subset R^1$.

\begin{lemma}\label{lemaideal} 
For all $n \gg 0$,
$$
R\, I_{p,q}^n = R^1 x^{pqn}.
$$
\end{lemma}
\begin{proof}
By~\eqref{property_Ipq}, the ideal $I_{p,q}^n$ is generated by all monomials $x^i y^j$ with $pi+qj \geq pqn$. 
Each monomial $x^i y^j$ is converted under $\varphi$ in $x^{pi+qj} y^{j} = x^{pi+qj-pqn} y^j \cdot x^{pqn}$ which 
belongs to $R^1 x^{pqn}$.

For the other inclusion, given $g = \sum_{i,j} a_{ij} x^i y^j \cdot x^{pqn} \in (\CC\{x\}[y])^{\qa_p} x^{pqn}$, 
consider as in the proof of \autoref{lema1},
$$
  h (x,y) = \sum_{i,j} a_{ij} x^{\frac{pqn+i-qj}{p}} y^j \ \in \ \CC\{x^{\pm 1}\}[y]
$$
which is an element of $\CC\{x\}[y]$ for all $n \gg 0$. The $(p,q)$-order of each monomial of $h$ is greater 
than or equal to $pqn$ hence they are in $I_{p,q}^n$ by~\eqref{property_Ipq}. Finally, it is clear that 
$\varphi(h+\langle f \rangle) = g + \langle \widehat{f} \, \rangle$ which concludes the proof.
\end{proof}

\begin{prop}\label{prop_one_step}
Using the previous conventions and assumptions, the order $\nu:=\nu_{p,q}(f)$ is a multiple of $pq$ and
$$
\dim_\CC \left(\frac{R^1}{R}\right) = \frac{\nu (\nu-p-q+1)}{2pq} \, \in \, \NN.
$$
\end{prop}

\begin{proof}
Since $E$ and $\widehat{C}$ only intersect at the origin of the first chart, the $(p,q)$-initial part 
of $f$ is of the form $f_{\nu} = \lambda y^b$, $\lambda \in \CC^{*}$; thus $q \,|\, \nu$. On the other 
hand, $f(x^p,x^q y) = x^{\nu} \widehat{f}(x,y)$ and $\widehat{f}(x,y)$ define functions on $X(p;-1,q)$ 
and hence $x^{\nu}$ is $\qa_p$-invariant. Consequently, $p \,|\, \nu$ and also $pq \,|\, \nu$ because 
$p$ and $q$ are coprime.

For the second part of the statement, by \autoref{lemaideal}, there is a short exact sequence, 
for all $n \gg 0$,
$$
0\longrightarrow \frac{R}{R I_{p,q}^n} \longrightarrow \frac{R^1}{R^1 x^{pqn}} 
\longrightarrow \frac{R^1}{R} \longrightarrow 0.
$$

The dimension of the first vector space is calculated in \autoref{lemma_dim_I+f}. 
The second one is a consequence of $E \cdot \widehat{C} = \frac{\nu}{pq}$, see 
\autoref{formula_self-intersection}\eqref{formula_self-intersection2},
\begin{equation*}
\dim_\CC \left( \frac{R^1}{R^1 x^{pqn}} \right) = 
\dim_\CC \left( \frac{\CC\{x,y\}^{\qa_p}}{\langle \widehat{f},x^{pqn}\rangle} \right) = pqn
\, E\cdot \widehat{C} = n \nu. \qedhere
\end{equation*}
\end{proof}

Now we are ready to state the main result of this section which allows one to interpret the 
invariant $\delta^w$ as the dimension of a vector space given by the normalization of the singularity. 

\begin{thm}
Let $f:(X(d;a,b),0)\to (\CC,0)$ be a reduced analytic function germ. Assume $(d;a,b)$ is a normalized type. 
Consider $R=\frac{\CC\{x,y\}^{\qa_d}}{\langle f\rangle }$ the local ring associated with $f$ and $\overline{R}$ 
its normalization ring. Then,
$$
\delta_0^\w(f) = \dim_\CC \left( \, \frac{\overline{R}}{R} \, \right) \, \in \, \NN.
$$
\end{thm}

\begin{proof}
After a suitable change of coordinates of the form $X(d;a,b)\to X(d;a,b)$, 
$[(x,y)]\mapsto [(x+\lambda y^k,y)]$ where $ bk\equiv a \mod d$, one can assume $x\nmid f$. 
Moreover, by \autoref{weierstrass_remark}, $f$ can be written in $\CC \{x,y\}^{\qa_d}$ in the form
\begin{equation}\label{special_form}
f(x,y) = y^r + \sum_{i\neq 0, \ j<r} a_{ij} x^i y^j \ \in \ \CC\{x\}[y].
\end{equation}

Consider $g\in \CC\{x,y\}$ the reduced germ obtained after applying \autoref{lema1} to~$f$. 
Denote by $R^{-1}=\frac{\CC\{x,y\}}{\langle g\rangle}$ its corresponding local ring and by 
$\pi_{(p,q)}$ the blowing-up at the origin with $p,q \in \NN$ coprime satisfying $X(p;-1,q)=X(d;a,b)$. 
Also, denote by $\delta_{\pi_{(p,q)}}$ the contributing rational number associated with the blowing-up
and to $f$,
$$
\delta_{\pi_{(p,q)}}(f) := \frac{\nu (\nu-p-q+1)}{2pq},
$$
where $\nu$ is $\nu_{p,q}(f)$.

On one hand $\delta_0(g)=\delta_{\pi_{(p,q)}}(f)+\delta_0^\w(f)$ by definition. On the other hand, 
since $\overline{R}$ is also the normalization of $g$ because $R^{-1}\subset R\subset \overline{R}$, 
by the classical case
$$
\delta_0(g) = \dim_\CC \left( \frac{\overline{R}}{R^{-1}} \right) = 
\dim_\CC \left( \, \frac{\overline{R}}{R} \, \right) + \dim_\CC \left( \frac{R}{R^{-1}} \right).
$$
The special form of $f(x,y)$ in~\eqref{special_form} implies the hypotheses of 
\autoref{prop_one_step} are satisfied for $g$. 
Hence $\delta_{\pi_{(p,q)}}(f) = \dim_\CC\frac{R}{R^{-1}}$ and the proof is complete.
\end{proof}

\section{A genus formula for weighted projective curves}
\label{sec-genus}

Denote by $\PP^2_\w$ the weighted projective space of weight $\w=(w_0,w_1,w_2)$ written in a normalized form,
that is, $\gcd(w_i,w_j)=1$, for any pair of weights $i\neq j$ (cf.~\cite{dolgachev-weighted}).  
Denote by $P_0:=[1:0:0]_\w$, $P_1:=[0:1:0]_\w$ and $P_2:=[0:0:1]_\w$ the three vertices of $\PP^2_\w$. 
Denote $\bw:=w_0w_1w_2$, $w_{ij}:=w_iw_j$ and $|\w|:=\sum_i w_i$. 

Note that the condition $\gcd(w_i,w_j)=1$ can be assumed without loss of generality, since a polynomial 
$F(X_0,X_1,X_2)$ defining a zero set on $\PP^2 \big(\w_0,\w_1,\w_2\big)$ such that $X_i\nmid F$ is transformed 
into another polynomial $F(X_0^{\frac{1}{d_0}},X_1^{\frac{1}{d_1}},X_2^{\frac{1}{d_2}})$ on $\PP^2(p_0,p_1,p_2)$
by the isomorphism defined in \autoref{propPw}.

Consider $\cC\subset \PP_\w^2$ a curve of degree $d$ given by a reduced equation $F=0$. 
We define $\Sing(\cC)=\cC\cap \left(\{\partial_x F=\partial_y F=\partial_z F=0\}\cup \{P_0,P_1,P_2\}\right)$. 
Thus, we say $\cC$ is \emph{smooth} if $\Sing(\cC)=\emptyset$. 
Also, we say $\cC$ is \emph{transversal w.r.t.~the axes} if $\{P_0,P_1,P_2\}\cap \cC=\emptyset$ and for any 
$P\in \cC\cap \{X_i=0\}$ the local equations of $\cC$ and $\{X_i=0\}$ are given by $uv=0$.

Formulas for the genus of quasi-smooth curves can be found in~\cite{dolgachev-weighted}.

Suppose $\cC\subset \PP_\w^2$ is a smooth curve of degree $d$ transversal w.r.t.~the axes. 
Consider the covering 
$$\begin{matrix}
   \phi:&\PP^2& \longrightarrow &\PP_w^2\\
        &[X_0:X_1:X_2]& \mapsto&[X_0^{w_0}:X_1^{w_1}:X_2^{w_2}]_\w\\  
\end{matrix} $$

Note that $\phi^{*}(\cC)$ is a smooth projective curve of degree $d$. 
Therefore $\chi(\phi^{*}(\cC))=2-(d-1)(d-2)$ and $\chi(\cC)=2-2g(\cC)$.
Using the Riemann-Hurwitz formula one obtains 
$$2-(d-1)(d-2)=\bw\left(\chi(\cC)-d\sum_i\frac{1}{w_{jk}}\right)+3d=2\bar \w-2\bw g(\cC)-d\sum_iw_{i}+3d$$
and hence 

$$2\bw g(\cC)=d^2 - d |\w| +2\bar \w.$$

This justifies the following.

\begin{dfn}
For a given $d\in \NN$ and a normalized weight list $\w\in \NN^3$ the \emph{virtual genus} associated with 
$d$ and $\w$ is defined as
$$
g_{d,\w}:=\frac{d(d - |\w|)}{2\bw}+1.
$$
\end{dfn}

\begin{rem}
\label{rem-virtual}
Note that $g_{d,\w}$ is always defined regardless of whether or not there actually exist even smooth curves of degree 
$d$ in $\PP^2_{\w}$. For instance, it is easy to see that there are no smooth curves of degree $5$ in $\PP^2_{(2,3,5)}$
(see \autoref{rem-smooth}). 

In general, by the discussion above, if $g_{d,\w}$ is not a positive integer, then no smooth curves in $\PP^2_{\w}$ of 
degree $d$ transversal w.r.t.~the axes can exist. However, this is not a sufficient condition, since $g_{40,(2,3,5)}=20$, 
but all curves of degree $40$ need to pass through at least one vertex.
\end{rem}

The characterization is given by the following.

\begin{lemma}
\label{lemma-generic}
Given $d$ and $\w$ as above, then the space of smooth curves of degree $d$ transversal w.r.t.~the axes in $\PP^2_{\w}$
is non-empty if and only if $\bw\, |\, d$. Moreover, any smooth curve can be deformed into a smooth curve of the same degree 
and transversal w.r.t.~the axes.
\end{lemma}

\begin{proof}
Let $F$ be a weighted homogeneous polynomial of degree $d$ whose set of zeroes defines $\cC$. The condition $P_i\notin \cC$ 
implies that $F$ contains a monomial of type $\lambda_iX_i^{d_i}$, $\lambda_i\neq 0$, $i=0,1,2$. Therefore $w_id_i=d$, which 
implies the result since $\gcd(w_i,w_j)=1$ by hypothesis on the weights~$\w$.

For the converse, assume $\bw | d$, then $X_0^{\frac{d}{w_0}}+X_1^{\frac{d}{w_1}}+X_2^{\frac{d}{w_2}}$ is a 
smooth curves of degree $d$ transversal w.r.t.~the axes in $\PP^2_{\w}$.

The \emph{moreover} part is a consequence of the fact that if $\cC$ is smooth then $\bw | d$ and hence 
$\cC + \lambda_0X_0^{\frac{d}{w_0}}+\lambda_1 X_1^{\frac{d}{w_1}}+\lambda_2 X_2^{\frac{d}{w_2}}$
is transversal w.r.t.~the axes for and appropriate generic choice of~$\lambda_i$.
\end{proof}

One has the following.

\begin{cor}
\label{cor-genus-smooth}
If $\cC$ is a smooth weighted curve in $\PP^2_\w$ of degree $d$, then $g(\cC)=g_{d,\w}$.
\end{cor}

\begin{proof}
By \autoref{lemma-generic} one can assume that $\cC$ is transversal w.r.t.~the axes. 
The result follows immediately from the discussion above.
\end{proof}

\begin{rem}
\label{rem-smooth}
Note that \autoref{cor-genus-smooth} does not apply to quasi-smooth weighted curves, that is, curves whose
equation is a weighted homogeneous polynomial whose only singularity in $\CC^3$ is $\{0\}$.
For instance, the curve $\cC:=\{X_0X_1=X_2\}\subset \PP^2_{(2,3,5)}$ of degree~$5$ can be parametrized by the map
$\PP^1 \to \PP^2_{(2,3,5)}$, given by $[t:s]\mapsto [t^2:s^3:t^2s^3]_\w$. Hence it is rational and $g(\cC)=g(\PP^1)=0$.
However, $g_{5,(2,3,5)}=\frac{7}{12}$. As a consequence of \autoref{cor-genus-smooth}, there are no
smooth curves of degree~5 in $\PP^2_{(2,3,5)}$ (see \autoref{rem-virtual}).
\end{rem}

The purpose of this section is to prove the following.

\begin{thm}
\label{thm-main}
Let $\cC\subset \PP^2_{\w}$ be a reduced curve of degree $d>0$, then 
$$
g(\cC) = g_{d,\w} - \sum_{P\in \Sing(\cC)} \delta^{\w}_P.
$$
\end{thm}

\begin{proof}
Let $F\in \CC[X_0,X_1,X_2]$ be a defining equation for $\cC$. Note that $F^{\bw}$ defines a function. 
Also, from the proof of \autoref{lemma-generic} one can obtain an algebraic 
family of smooth curves $\cC_t=\{F_t=0\}$, $t\in (0,1]$ of degree $d\bw$ whose defining polynomials 
$F_t$ degenerate to $F^{\bw}=F_0$ such that $\cC_t$ and $\cC$ intersect transversally and outside the 
axes. Therefore, by \autoref{cor-genus-smooth}, 
$$g(\cC_t)=g_{w,d\bw}=\frac{d\bw}{2\bw}(d\bw - |\w|)+1=\frac{d}{2}(d\bw - |\w|)+1.$$
On the other hand, define $I_t:=\cC \cap \cC_t$. Using Bézout's \autoref{thm-bezout-P2w}, 
$$
(\cC\cdot \cC_t)=\sum_{P_t\in I_t} (\cC\cdot \cC_t)_{P_t} = \frac{1}{\bw}d^2 \bw=d^2.
$$
Since $\cC_t$ and $\cC$ intersect transversally outside the axes at smooth points 
one has $(\cC\cdot \cC_t)_{P_t}=1$ (see \autoref{exam-nu-nc}) and hence $\# I_t=d^2$.

For each $P\in \Sing(\cC)$, consider $B^S_P$ a regular neighborhood of $\cC$ at $P$ and for each
$P_t\in I_t$ consider $B^I_{P_t}$ a regular neighborhood of $\cC\cup \cC_t$ at $P_t$.

Note that, outside $B:=\bigcup_{P\in \Sing(\cC)} B^S_P \cup \bigcup_{P_t\in I_t} B^I_{P_t}$, 
$\cC_t$ provides a $\bw:1$ covering of $\cC$, that is,
$$
\chi(\cC_t \setminus B\cap \cC_t) = \bw \cdot \chi(\cC\setminus B\cap \cC).
$$
Also, in each $B^S_P$, the curve $\cC_t$ is the disjoint union of $\bw$ Milnor fibers of $(\cC,P)$. 
Therefore
\begin{equation}
\label{eq-chi0}
\array{c}
\chi(\cC_t) = \chi(\cC_t \setminus B\cap \cC_t) + \sum_{P\in \Sing(\cC)} \chi(B^S_P\cap \cC_t) + 
\sum_{P_t\in I_t} \chi(B^I_{P_t}\cap \cC) = \\
\bw \cdot \chi(\cC \setminus B\cap \cC) + 
\bw \cdot \left[ \sum_{P\in \Sing(\cC)} 2\delta^\w_P + \sum_{P\in \Sing(\cC)} r_P \right] + d^2.
\endarray
\end{equation}
On the other hand
\begin{equation}
\label{eq-chi1}
\chi(\cC_t) = 2- 2\left(\frac{d}{2}(d\bw - |\w|)+1\right)=d|\w| - d^2\bw
\end{equation}
and
$$
\chi(\cC) = 
\begin{cases}
2 - 2g(\cC) - \sum_{P\in \Sing(\cC)} (r_P-1) \\
\chi(\cC \setminus B\cap \cC) + d^2 + \# \Sing(\cC).
\end{cases}
$$
Therefore
\begin{equation}
\label{eq-chi2}
\chi(\cC \setminus B\cap \cC) = 2 - 2g(\cC) - d^2 - \sum_{P\in \Sing(\cC)} r_P.
\end{equation}
Substituting \eqref{eq-chi1} and \eqref{eq-chi2} in \eqref{eq-chi0} one obtains
$$
d|\w| - d^2\bw = \bw \cdot \left[ 2 - 2g(\cC) - d^2 - \sum_{P\in \Sing(\cC)} r_P + 
\sum_{P\in \Sing(\cC)} 2\delta^\w_P + \sum_{P\in \Sing(\cC)} r_P \right] + d^2,
$$
which after simplification becomes 
$$
2 \bw g(\cC)=d^2 - d|\w| + 2 \bw - 2 \bw \cdot \left( \sum_{P\in \Sing(\cC)} \delta^\w_P\right)
$$
and results into the desired formula.
\end{proof}

\begin{exam}
Let us consider the curve $\cC=\{X_0X_1-X_2\}\subset \PP^2_{\w}$, with $\w=(a,b,a+b)$. 
Note that $\PP^1\to \PP^2_{\w}$ given by $[t:s]\mapsto [t^a:s^b:t^as^b]$ is an
isomorphism and hence $g(\cC)=g(\PP^1)=0$ (see discussion in \autoref{rem-smooth}).
In order to use \autoref{thm-main} one needs to compute the virtual genus of $\cC$
$$
g_{d,\w}=\frac{2ab-a-b}{2ab}.
$$
On the other hand $\Sing(\cC)=\{P_0,P_1\}$. Note that both singularities $P_0$ and $P_1$
are of type $x^p-y^q$ with $(p,q)=(1,1)$ in their respective quotient-singularity charts
($P_0\in X(a;b,a+b)$ and $P_1\in X(b;a,a+b)$) and thus, 
formula~\eqref{eq-pq} implies:
$$\delta^\w_{P_0}=\frac{1}{2} \frac{(1-1-1+a)}{a}=\frac{a-1}{2a}$$
and
$$\delta^\w_{P_1}=\frac{1}{2} \frac{(1-1-1+b)}{b}=\frac{b-1}{2b}.$$
Therefore, according to \autoref{thm-main}
$$
g(\cC)=g_{d,\w}-\delta^\w_{P_0}-\delta^\w_{P_1}=
\frac{2ab-a-b}{2ab} - \frac{a-1}{2a} - \frac{b-1}{2b} = 0.
$$
\end{exam}

\begin{exam}
Let us consider now the curve $\cC=\{X_0X_1-X_2^2\}\subset \PP^2_{\w}$, with $\w=(2k-1,2k+1,2k)$. 
In order to use \autoref{thm-main} one needs to compute the virtual genus of $\cC$
$$
g_{4k,\w}=\frac{4k(4k-6k)}{2(4k^2-1)2k}+1=1-\frac{2k}{(4k^2-1)}.
$$
On the other hand $\Sing(\cC)=\{P_0,P_1\}$, $P_0\in X(2k-1;2k+1,2k)$ and $P_1\in X(2k+1;2k-1,2k)$.

Using \autoref{deltapqab} one has, 
$$\delta_{P_0}^\w=\frac{2-1-2+(2k-1)}{2(2k-1)}=\frac{k-1}{2k-1}$$
and
$$\delta_{P_1}^\w=\frac{2-1-2+(2k+1)}{2(2k+1)}=\frac{k}{2k+1}.$$

Therefore, according to \autoref{thm-main}
$$
g(\cC)=1-\frac{2k}{(4k^2-1)}-\frac{k-1}{2k-1}-\frac{k}{2k+1}=0.
$$

\end{exam}

\begin{exam}

Let us consider the curve $\cC=\{X_0X_1X_2+(X_0^3-X_1^2)^2\}\subset \PP^2_{\w}$ of quasi-homogeneous degree $d=12$, 
with $\w=(2,3,7)$. 
Note that 
$$
g_{12,\w}=\frac{12(12-12)}{2\bar{\w}}+1=1.
$$
On the other hand, $\Sing(\cC)=\{P_2\}$, which is a quotient singularity of local type $xy+(x^2-y^3)^2$ in $X(7;2,3)$. 
In order to obtain $\delta_{P_2}^\w$ one can perform, for instance, a blow-up of type $(1,5)$. The multiplicity of the 
exceptional divisor is $6$ and hence
\begin{equation}
\label{eq-exdelta1}
\frac{\nu(\nu-p-q+e)}{2dpq}=\frac{6(6-1-5+7)}{2\cdot 7\cdot 1\cdot 5}=\frac{3}{5}.
\end{equation}
After this first blow-up, the two branches separate and the strict preimage becomes a smooth 
branch (at a smooth point of the surface) and a singularity of type $x^p-y^q$, where $(p,q)=(1,14)$, in 
$X(5;2,1)$. Using formula~\eqref{eq-pq} one obtains:
\begin{equation}
\label{eq-exdelta2}
\frac{\nu(\nu-p-q+e)}{2dpq}=\frac{pq-p-q+d}{2d}=\frac{14-1-14+5}{2\cdot 5}=\frac{2}{5}.
\end{equation}
Combining~\eqref{eq-exdelta1} and \eqref{eq-exdelta2} one obtains
$$\delta_{P_2}^\w=\frac{3}{5}+\frac{2}{5}=1.$$
Therefore $g(\cC)=1-1=0$ according to \autoref{thm-main}.
\end{exam}

\end{document}